\documentclass[12pt]{article}

\usepackage{amsmath,amssymb,amsbsy,amsfonts,amsthm,latexsym,
                     amsopn,amstext,amsxtra,euscript,amscd}

\begin{document}

\newtheorem{theorem}{Theorem}
\newtheorem{lemma}[theorem]{Lemma}
\newtheorem{claim}[theorem]{Claim}
\newtheorem{cor}[theorem]{Corollary}
\newtheorem{prop}[theorem]{Proposition}
\newtheorem{definition}{Definition}
\newtheorem{question}[theorem]{Open Question}

\def\cA{{\mathcal A}}
\def\cB{{\mathcal B}}
\def\cC{{\mathcal C}}
\def\cD{{\mathcal D}}
\def\cE{{\mathcal E}}
\def\cF{{\mathcal F}}
\def\cG{{\mathcal G}}
\def\cH{{\mathcal H}}
\def\cI{{\mathcal I}}
\def\cJ{{\mathcal J}}
\def\cK{{\mathcal K}}
\def\cL{{\mathcal L}}
\def\cM{{\mathcal M}}
\def\cN{{\mathcal N}}
\def\cO{{\mathcal O}}
\def\cP{{\mathcal P}}
\def\cQ{{\mathcal Q}}
\def\cR{{\mathcal R}}
\def\cS{{\mathcal S}}
\def\cT{{\mathcal T}}
\def\cU{{\mathcal U}}
\def\cV{{\mathcal V}}
\def\cW{{\mathcal W}}
\def\cX{{\mathcal X}}
\def\cY{{\mathcal Y}}
\def\cZ{{\mathcal Z}}

\def\A{{\mathbb A}}
\def\B{{\mathbb B}}
\def\C{{\mathbb C}}
\def\D{{\mathbb D}}
\def\E{{\mathbb E}}
\def\F{{\mathbb F}}
\def\G{{\mathbb G}}
\def\I{{\mathbb I}}
\def\J{{\mathbb J}}
\def\K{{\mathbb K}}
\def\L{{\mathbb L}}
\def\M{{\mathbb M}}
\def\N{{\mathbb N}}
\def\O{{\mathbb O}}
\def\P{{\mathbb P}}
\def\Q{{\mathbb Q}}
\def\R{{\mathbb R}}
\def\S{{\mathbb S}}
\def\T{{\mathbb T}}
\def\U{{\mathbb U}}
\def\V{{\mathbb V}}
\def\W{{\mathbb W}}
\def\X{{\mathbb X}}
\def\Y{{\mathbb Y}}
\def\Z{{\mathbb Z}}

\def\ep{{\mathbf{e}}_p}

\def\scr{\scriptstyle}
\def\\{\cr}
\def\({\left(}
\def\){\right)}
\def\[{\left[}
\def\]{\right]}
\def\<{\langle}
\def\>{\rangle}
\def\fl#1{\left\lfloor#1\right\rfloor}
\def\rf#1{\left\lceil#1\right\rceil}
\def\le{\leqslant}
\def\ge{\geqslant}
\def\eps{\varepsilon}
\def\mand{\qquad\mbox{and}\qquad}

\def\vec#1{\mathbf{#1}}
\def\inv#1{\overline{#1}}
\def\vol#1{\mathrm{vol}\,{#1}}

\newcommand{\comm}[1]{\marginpar{%
\vskip-\baselineskip 
\raggedright\footnotesize
\itshape\hrule\smallskip#1\par\smallskip\hrule}}

\def\xxx{\vskip5pt\hrule\vskip5pt}


\title{\bf Visible Points on Curves over Finite Fields}

\author{
{\sc Igor E. Shparlinski} \\
{Department of Computing, Macquarie University} \\
{Sydney, NSW 2109, Australia} \\
{igor@ics.mq.edu.au} \\
{\sc Jos{\'e} Felipe Voloch} \\
{Department of Mathematics,
University of Texas}\\
{Austin TX 78712 USA} \\
{voloch@math.utexas.edu}
}

\date{\today}
\pagenumbering{arabic}

\maketitle

\begin{abstract} For a prime $p$ and an absolutely irreducible modulo $p$
polynomial $f(U,V) \in \Z[U,V]$ we obtain an asymptotic
formulas  for the number of solutions
to the congruence $f(x,y) \equiv a \pmod p$ in positive integers
$x \le X$, $y \le Y$, with the additional condition
$\gcd(x,y)=1$. Such solutions have a natural interpretation
as solutions  which are visible from the origin.
 These formulas are 
derived on average over $a$ for a fixed prime $p$, and also 
on average over $p$ for a fixed integer  $a$.
\end{abstract}


\section{Introduction}
\label{sec:intro}

Let  $p$ be a prime and let $f(U,V) \in \Z[U,V]$
be a bivariate
polynomial with integer coefficients.

For real  $X$ and $Y$ with
$1 \le X,Y \le p$ and an integer $a$ we consider the set
$$
\cF_{p,a}(X,Y)
=   \{(x,y) \in  [ 1,  X]\times [1,Y] \ :\  f(x,y) \equiv a \pmod p\}
$$
which the set of points on level curves of $f(U,V)$
modulo $p$.

If the polynomial $f(x,y) -a$ is nonconstant absolutely irreducible
polynomial modulo $p$ of degree bigger than one can easily derive from
the Bombieri bound~\cite{Bomb} that
\begin{equation}
\label{eq:Bomb Bound}
\# \cF_{p,a}(X,Y)  = \frac{XY}{p} + O\(p^{1/2} (\log p)^2\),
\end{equation}
where the implied constant depends only on $\deg f$, see,
for example,~\cite{CoZa,GrShZa,VaZa,Zhe}.

In this paper we consider an apparently new question of studying  the
set
$$N_{p,a}(X,Y)
=   \{(x,y) \in\cF_{p,a}(X,Y)\ : \ \gcd(x,y)=1\}.
$$
These points have a natural geometric interpretation as points
on $\cF_{p,a}(X,Y)$ which are ``visible'' from the origin,
see~\cite{BCZ,HuNo,Now,Zhai}
and references therein for several other aspects of distribution of
visible points in various regions.

We show that on average over $a=0, \ldots, p-1$,
the cardinality  $N_{p,a}(X,Y)$  is close to its
expected value $6XY/\pi^2p$, whenever
\begin{equation}
\label{eq:Nontriv}
XY \ge p^{3/2 +\varepsilon}
\end{equation}
for any fixed $\varepsilon> 0$ and sufficiently large $p$.

We then consider the dual situation, when $a$ is fixed (in
particular we take $a=0$) but $p$ varies through all
primes up to $T$.

We recall  $A \ll B$ and  $A = O(B)$
both mean that $|A| \le c B$ holds with some
constant $c> 0$, which may depend on some
specified set of parameters.

\section{Absolute Irreducibility of Level Curves}

We start with the following statement which could
be of independent interest.

\begin{lemma}
\label{lem:Irred} If $F(U,V) \in \K[U,V]$ is absolutely irreducible
of degree $n$ over a field $\K$, then $F(U,V)-a$
is absolutely irreducible for all but
at most  $C(n)$
elements  $a \in \K$, where $C(n)$ depends only on $n$.
\end{lemma}

\begin{proof}
The set of polynomials of degree $n$ is parametrized by a
projective space $\P^{s(n)}$ of dimension $s(n) = (n+1)(n+2)/2$
over $\K$,
   coordinatized by the coefficients. The subset
$X$ of $\P^{k(n)}$ consisting of reducible polynomials is a Zariski closed
subset because it is the union of
the images of the maps
$$\P^{s(k)}\times \P^{s(n-k)} \to \P^{s(n)}, \qquad  k \le  n/2,$$
given by multiplying a polynomial of degree $k$ with a polynomial of
degree $n-k$.
The map $t \mapsto F(U,V)-t$ describes a line in $\P^{s(n)}$ and
by the assumption of absolutely irreducibility of $F$,
    this line is not contained in $X$.
So, by the B\'ezout theorem, it meets $X$ in at most
$C(n)$ points, where $C(n)$ is the degree of $X$.
Hence for all but at most $C(n)$ values of $a$, $F(U,V)-a$ is
absolutely irreducible.
\end{proof}

\section{Visible Points on Almost All Level Curves}

Throughout this section, the implied constants in
    the notations $A \ll B$ and  $A = O(B)$ may depend on
the degree $n = \deg f$.

\begin{theorem}
\label{thm:Visible} Let $f$ be a polynomial with integer coefficients
which is absolutely irreducible and of degree bigger than one modulo 
the prime $p$. Then
for real  $X$ and $Y$ with
$1 \le  X, Y \le  p$
we have
$$
\sum_{a=0}^{p-1} \left|N_{p,a}(X,Y)- \frac{6}{\pi^2} \cdot \frac{XY}{p} \right|
\ll  X^{1/2} Y^{1/2} p^{3/4}  \log p.
$$
\end{theorem}

\begin{proof}
Let $\cA_p$ consist of $a \in \{0,\ldots, p-1\}$ for which
$f(U,V)-a$ is absolutely irreducible modulo $p$.

    For an integer $d$,  we define
$$
     M_{p,a}(d;X,Y) = \# \{(x,y) \in \cF_{p,a}(X,Y) \ | \ \gcd(x,y)
\equiv 0 \pmod d\}.
$$

Let  $\mu(d)$ denote the M\"obius function.
    We recall that $\mu(1) = 1$, $\mu(d) = 0$ if $d \ge 2$ is not
square-free  and $\mu(d) = (-1)^{\omega(d)}$ otherwise, where $\omega(d)$
is the number of distinct prime divisors $d$.
By the inclusion-exclusion principle, we write
\begin{equation}
\label{eq:N and Md}
N_{p,a}(X,Y) = \sum_{d=1}^\infty \mu(d)  M_{p,a}(d;X,Y) .
\end{equation}

Writing
$$
x = ds\mand y = dt,
$$
we have
$$
\#  M_{p,a}(d;X,Y)   =   \# \{(s,t) \in [1,  X/d]\times [1, Y/d]\ | \
f(ds, dt)  \equiv a
\pmod p\}.
$$

Thus $M_{p,a}(d;X,Y)$ is the number of points on a curve
in a given box.
If $a \in \cA_p$ and $1\le d < p$ then
$f(dU, dV) - a$ remains absolutely irreducible modulo $p$.
Accordingly, we have an analogue
of~\eqref{eq:Bomb Bound} which asserts
that
\begin{equation}
\label{eq:Small d}
    M_{p,a}(d;X,Y) = \frac{XY}{d^2 p} + O\(p^{1/2} (\log p)^2\).
\end{equation}
We fix some  positive parameter $D<p$ and substitute the
bound~\eqref{eq:Small d}
in~\eqref{eq:N and Md}
for $d\le D$, getting
\begin{eqnarray*}
\lefteqn{N_{p,a}(X,Y) }\\
& & \qquad =
\sum_{d \le D} \(\frac{\mu(d)XY}{d^2 p} + O\(p^{1/2} (\log p)^2\)\)
+ O\(\sum_{d > D} M_{p,a}(d;X,Y)\)\\
& & \qquad =
\frac{XY}{p} \sum_{d \le D} \frac{\mu(d)}{d^2} + O\(Dp^{1/2} (\log
p)^2+\sum_{d > D}
M_{p,a}(d;X,Y)\)
\end{eqnarray*}
for every $a\in \cA_p$.

Furthermore
$$
\sum_{d \le D} \frac{\mu(d)}{d^2} = \sum_{d =1}^\infty
\frac{\mu(d)}{d^2} + O(D^{-1})
= \prod_{\ell} \(1 - \frac{1}{\ell^2}\) + O(D^{-1}),
$$
where the product is taken over all prime numbers $\ell$.
Recalling that
$$
\prod_{\ell} \(1 - \frac{1}{\ell^2}\)  = \zeta(2)^{-1} = \frac{6}{\pi^2},
$$
see~\cite[Equation~(17.2.2) and Theorem~280]{HaWr}, we
obtain
\begin{equation}
\label{eq:N prelim}
    \left|N_{p,a}(X,Y)- \frac{6}{\pi^2} \cdot \frac{XY}{p} \right|
\ll XY/Dp + Dp^{1/2} (\log p)^2+\sum_{d > D}
M_{p,a}(d;X,Y),
\end{equation}
for every $a\in \cA_p$.

    We also remark that
\begin{equation}
\begin{split}
\label{eq:Large d}
    \sum_{a=0}^{p-1} \sum_{d > D} M_{p,a}(d;X,Y)  &=
\sum_{d >D}\sum_{a=0}^{p-1}  M_{p,a}(d;X,Y)
\\
&
= \sum_{d > D} \fl{\frac{X}{d}} \fl{\frac{Y}{d}}
\le XY \sum_{d > D} \frac{1}{d^2} \ll XY/D.
\end{split}
\end{equation}
Therefore, using
the bounds~\eqref{eq:N prelim} and~\eqref{eq:Large d},
we obtain
\begin{equation}
\label{eq:good a}
\sum_{a \in \cA_p} \left|N_{p,a}(X,Y)- \frac{6}{\pi^2} \cdot
\frac{XY}{p} \right|
\ll XY/D + Dp^{3/2} (\log p)^2.
\end{equation}
For $a\not\in \cA_p$ we estimate $N_{p,a}(X,Y)$ trivially
as
$$
N_{p,a}(X,Y) \le \min\{X, Y\} \deg f \ll \sqrt{XY}.
$$
Thus by  Lemma~\ref{lem:Irred},
\begin{equation}
\label{eq:bad a}
\sum_{a \not \in \cA_p} \left|N_{p,a}(X,Y)- \frac{6}{\pi^2} \cdot
\frac{XY}{p} \right|
\ll \max\{\sqrt{XY}, XY/p\} \ll \sqrt{XY}.
\end{equation}
Combining~\eqref{eq:good a} and~\eqref{eq:bad a} and
taking $D = X^{1/2} Y^{1/2} p^{-3/4}  (\log p)^{-1}$
we conclude the proof.
\end{proof}

\begin{cor}
\label{cor:almost all a}  Let $f$ be a polynomial with integer 
coefficients which is absolutely irreducible and of degree bigger than one.
 If  $XY \ge p^{3/2} (\log
p)^{2+\varepsilon}$ for some
fixed $\varepsilon > 0$, then
$$
N_{p,a}(X,Y)= \(\frac{6}{\pi^2} + o(1)\) \frac{XY}{p}
$$
for all but $o(p)$ values of $a=0, \ldots, p-1$.
\end{cor}

\section{Visible Points on Almost All Reductions}

Throughout this section, the implied constants in
    the notations $A \ll B$ and  $A = O(B)$ may depend on
the coefficients of $f$.

To simplify notation we put
$$
\cF_{p}(X,Y) =   \cF_{p,0}(X,Y)\mand
N_{p}(X,Y)
=  N_{p,0}(X,Y).
$$

\begin{theorem}
\label{thm:Visible_overp} Let $f$ be a polynomial with integer 
coefficients which is absolutely irreducible and of degree bigger than one.
Then for real  $T$, $X$ and $Y$ such that $T \ge 2\max(X,Y)$,
we have
$$
\sum_{T/2 \le p\le T} \left|N_{p}(X,Y)- \frac{6}{\pi^2} \cdot
\frac{XY}{p} \right|
\ll   X^{1/2} Y^{1/2} T^{3/4+o(1)},
$$
where the sum is taken over all primes $p$ with $T/2 \le p\le T$.
\end{theorem}

\begin{proof} It is enough to consider $T$ large enough so that
$f$ remains absolutely irreducible and of degree bigger than one
for all $p, T/2 \le p\le T$.
As before we have
\begin{equation}
\label{eq:N prelim'}
    \left|N_{p}(X,Y)- \frac{6}{\pi^2} \cdot \frac{XY}{p} \right|
\ll XY/Dp + Dp^{1/2} (\log p)^2+\sum_{d > D} M_{p}(d;X,Y).
\end{equation}
where
$$
     M_{p}(d;X,Y) = \# \{(x,y) \in \cF_{p}(X,Y) \ | \ \gcd(x,y) \equiv
0 \pmod d\}  .
$$

    We also remark that
\begin{equation}
\begin{split}
\label{eq:Large d'}
    \sum_{T/2 \le p\le T} \sum_{d > D} M_{p}(d;X,Y)  &=
\sum_{d >D}\sum_{T \le p\le T} M_{p}(d;X,Y)
\\
&
= \sum_{d > D} \sum_{1 \le s \le X/d}\sum_{1 \le t \le Y/d}
\sum_{\substack{T/2 \le p\le T\\ p|f(ds,dy)}} 1.
\end{split}
\end{equation}

Let $\cZ$ be set of integer zeros of $f$ in the relevant box, that is
$$
\cZ = \{(u,v) \in \Z^2\ : 1 \le x \le X, 1 \le y \le Y,\ f(u,v) = 0\}.
$$
It is easy to see that $\# \cZ \ll \min(X,Y) \le \sqrt{XY}$.
Indeed, it is enough to notice that
since $f(U,V)$ is absolutely irreducible, each
specialization
$g_y(U) = f(U,y)$ with $y \in \Z$ and $h_x(V) = f(x,V)$ with $x \in \Z$ 
is a nonzero polynomials in $U$ and $V$, respectively.
(Under extra, but generic, hypotheses, one can invoke Siegel's theorem, 
which gives $\# \cZ = O(1)$ but this does not lead to an improvement in 
our final bound.)
Denoting by $\tau(k)$ the number of integer divisors of  a positive
integer $k$, we see that
for each $(u,v)\in \cZ$ there are at most $\tau(u) = X^{o(1)}$
(see~\cite[Theorem~317]{HaWr}) pairs  $(d,s)$ of positive
integers
with $ds=u$, after which there is at most one
value of $t$.  Thus for these triples
$(d,s,t)$,  we estimate the inner sum
over $p$ in~\eqref{eq:Large d'} trivially as $T$.

To estimate the rest of the sums, as before, we denote by
$\omega(k)$ the number of prime divisors of  a positive
integer $k$ and note that $\omega(k) \ll \log k$. Thus for $(u,v)\not \in \cZ$
we can estimate the inner sum
over $p$ in~\eqref{eq:Large d'} as $\omega(|f(ds,dy)|) =
  (XY)^{o(1)}$.
Therefore
\begin{eqnarray*}
\sum_{T/2 \le p\le T} \sum_{d > D} M_{p}(d;X,Y)
&\le & \sum_{d > D} \sum_{\substack{1 \le
s \le X/d\\1 \le t \le Y/d\\ (ds,dt)\in \cZ}}\,
\sum_{T/2 \le p\le T} 1
+
\sum_{d > D} \sum_{\substack{1 \le
s \le X/d\\1 \le t \le Y/d\\ (ds,dt)\not \in \cZ}}\,
\sum_{ p|f(ds,dy)} 1\\
& \le  & \# \cZ X^{o(1)} T
+ (XY)^{o(1)}
\sum_{d > D} \sum_{\substack{1 \le
s \le X/d\\1 \le t \le Y/d\\ (ds,dt)\not \in \cZ}} 1\\
& = &  (XY)^{1/2+o(1)} T
+ (XY)^{1+ o(1)} D^{-1}.
\end{eqnarray*}

We now put everything together getting
\begin{align*}
\sum_{T/2 \le p\le T}  & \left|N_{p}(X,Y)- \frac{6}{\pi^2} \cdot 
\frac{XY}{p} \right|
\\ &\ll XY/D\log T  + DT^{3/2} (\log T)^2+  T (XY)^{1/2+o(1)} + (XY)^{1+o(1)} D^{-1},
\end{align*}
and take  $D = X^{1/2} Y^{1/2} T^{-3/4}$
getting the result.
\end{proof}

\begin{cor}
\label{cor:almost all p}
 Let $f$ be a polynomial with integer 
coefficients which is absolutely irreducible and of degree 
bigger than one.
If  $XY \ge T^{3/2+\varepsilon}$ for some
fixed $\varepsilon > 0$, that
$$
N_{p}(X,Y)= \(\frac{6}{\pi^2} + o(1)\) \frac{XY}{p}
$$
for all but $o(T/\log T)$ primes $p \in [T/2,T]$.
\end{cor}

\section{Remarks}

Certainly it would be interesting to obtain  an asymptotic formula
for $N_{p,a}(X,Y)$ which holds for every $a$. 
Even the case of $X = Y = p$ would be of interest.
We remark that for the polynomial  
$f(U,V) = UV$ such an asymptotic formula is give
in~\cite{Shp} and is nontrivial provided 
$XY \ge p^{3/2+\varepsilon}$ for some
fixed $\varepsilon > 0$. 
However the technique of~\cite{Shp} does not seem to
apply to more general polynomials. 

We remark that studying such special cases as 
visible points on the curves of the shape 
$f(U,V) = V - g(U)$ (corresponding to points a graph 
of a univariate polynomial) and 
$f(U,V) = V^2 - X^3 -rX - s$ (corresponding to points 
on an elliptic curve) is also of interest and may
be more accessible that the general case.

\section*{Acknowledgements.}  This work began
during a pleasant visit by I.~S.\ to University of Texas
sponsored by NSF grant DMS-05-03804; the
support and hospitality of this institution are gratefully
acknowledged. During the preparation of this paper, I.~S.\ was
supported in part by ARC grant DP0556431.

\end{document}